\documentclass[11pt,a4paper]{amsart}
\usepackage[T1]{fontenc}
\usepackage{lmodern}
\usepackage{amsmath,amssymb,mathtools,booktabs}
\usepackage{microtype}
\usepackage[a4paper,textwidth=155mm,textheight=235mm,centering]{geometry}
\usepackage[hidelinks]{hyperref}
\newtheorem{theorem}{Theorem}[section]
\newtheorem{maintheorem}{Theorem}

\newtheorem{proposition}[theorem]{Proposition}
\newtheorem{lemma}[theorem]{Lemma}
\newtheorem{corollary}[theorem]{Corollary}
\theoremstyle{remark}
\newtheorem{remark}[theorem]{Remark}
\newcommand{\F}{\mathbb F}
\newcommand{\N}{\mathbb Z_{\ge0}}
\newcommand{\Cart}{\mathcal C}
\newcommand{\divis}{\operatorname{div}}
\newcommand{\HH}{H^0}
\numberwithin{equation}{section}
\title[$a$-numbers of covers in characteristic two]{$a$-numbers of covers in characteristic two:\\ semigroups and Riemann--Roch spaces}
\author{Gilberto B. Almeida Filho}
\address{Departamento de Ci\^encias Exatas e da Terra, Universidade Federal de Mato Grosso, V\'arzea Grande, Mato Grosso, Brazil}
\email{gilberto.filho@ufmt.br}
\author{Saeed Tafazolian}
\address{IMECC, Universidade Estadual de Campinas, Campinas, S\~ao Paulo, Brazil}
\email{saeed@unicamp.br}
\author{St\'efani C. Vieira}
\address{Departamento de Ci\^encias Exatas e da Terra, Universidade Federal de Mato Grosso, V\'arzea Grande, Mato Grosso, Brazil}
\email{stefani.vieira1@ufmt.br}
\date{}
\subjclass[2020]{Primary 14H40; Secondary 11G20, 14H55}
\keywords{Cartier operator, a-number, Weierstrass semigroup, cyclic cover, maximal curve}
\begin{document}
\begin{abstract}
We compute the $a$-numbers of the Garcia--G\"uneri--Stichtenoth curves, the Beelen--Montanucci curves, and the Skabelund covers of Suzuki curves in characteristic two. The proofs use the identity $a(X)=\ell(\tfrac12\operatorname{div}(dt))$ and reduce the calculations to Riemann--Roch spaces on the base curves. We obtain transfer formulas for odd-degree radical covers, with one or several points at infinity, and evaluate them using Weierstrass semigroups and plane sections of the Hermitian curve. For the first two families we also give cyclic models over the projective line and relate the formulas to Elkin's rank formula. We then treat two classes of Artin--Schreier fibre products. A supersingular genus-six family has $a$-number $3$ at the zero parameter and $2$ elsewhere, although the genera and $a$-numbers of its quadratic quotients remain fixed. We determine its Ekedahl--Oort types as $(0,0,0,1,2,3)$ and $(0,1,1,2,3,4)$, respectively.
\end{abstract}
\maketitle
\enlargethispage{2pt}

\section{Introduction}
Let $k=\overline{\F}_2$. For a smooth projective connected curve $X$ over $k$, the $a$-number of its Jacobian is
\[
a(X)=\dim_k\ker\bigl(\Cart:\HH(X,\Omega_X^1)\longrightarrow\HH(X,\Omega_X^1)\bigr),
\]
where $\Cart$ denotes the Cartier operator. This invariant contains information about the group scheme $\operatorname{Jac}(X)[2]$ which is not determined by the Newton polygon.

We calculate this invariant for several families of covers by counting sections of half-canonical divisors. Suppose $Y$ has a function $f$ with pole divisor $NQ$ and $\divis(df)=(2g_Y-2)Q$. For a positive odd integer $m$ coprime to $N$, Theorem~\ref{thm:transfer} determines the semigroup and $a$-number of the radical cover $z^m=f$ by
\[
H_X(P)=mH_Y(Q)+N\N,\qquad
a(X)=\sum_{j=0}^{m-1}\ell_Y\!\left(\left\lfloor\frac{g_X-1-jN}{m}\right\rfloor Q\right).
\]
The same argument gives a multipoint formula (Theorem~\ref{thm:multitransfer}). For covers unramified away from one point, Proposition~\ref{prop:wildpreserve} also applies in the presence of wild ramification.

\subsection*{Main results}
Giulietti and Korchm\'aros \cite{GK} constructed a maximal curve over $\F_{q^6}$ which, for $q>2$, is not covered by the Hermitian curve maximal over the same field. Garcia, G\"uneri, and Stichtenoth \cite{GGS} extended this construction to an infinite family. For $q=2^h$, $h\ge1$, and odd $n\ge3$, let $\mathcal G_{q,n}$ be the smooth projective model of
\begin{equation}\label{eq:GGSmodel}
x^q+x=y^{q+1},\qquad z^{d_n}=y^{q^2}+y,
\qquad d_n=\frac{q^n+1}{q+1}.
\end{equation}
The case $n=3$ is the Giulietti--Korchm\'aros curve. The genus is
\[
g(\mathcal G_{q,n})=\frac{(q-1)(q^{n+1}+q^n-q^2)}2.
\]
The following theorem records the resulting formula for this family.

\begin{maintheorem}\label{thm:GGS}
Let $q=2^h$, $h\ge1$, and let $n\ge3$ be odd. For $n=3$,
\[
a(\mathcal G_{q,3})=\frac{q(q+1)(3q^3-6q^2+10q-4)}{24}.
\]
For $n\ge5$,
\begin{equation}\label{eq:GGSanswer}
a(\mathcal G_{q,n})=
\frac{q^{n-3}(3q^6+q^4+2q^3-q^2-2q)-q^3+4q}{24(q+1)}.
\end{equation}
For each fixed odd $n\ge3$, the ratio $a(\mathcal G_{q,n})/g(\mathcal G_{q,n})$ tends to $1/4$ as $q=2^h$ tends to infinity.
\end{maintheorem}

Another application concerns the cyclic cover of the Suzuki curve constructed by Skabelund \cite{Skabelund}. Here let $q_0=2^s$ and $q=2q_0^2$, with $s\ge1$, and let $X_q$ have affine equations
\begin{equation}\label{eq:model}
y^q+y=x^{q_0}(x^q+x),\qquad t^{q-2q_0+1}=x^q+x.
\end{equation}
This curve is maximal over $\F_{q^4}$, and
\begin{equation}\label{eq:genus}
g(X_q)=\frac{q(q-1)^2}{2}.
\end{equation}

\begin{maintheorem}\label{thm:main}
Let $X_q$ be the smooth projective curve defined by \eqref{eq:model}, where $q_0=2^s$, $q=2q_0^2$, and $s\ge1$. Then
\begin{equation}\label{eq:answer}
a(X_q)=\frac{3q_0^6-3q_0^5+2q_0^4+q_0^3+q_0^2-q_0}{3}.
\end{equation}
Consequently,
\[
\operatorname{rank}\Cart
=\frac{9q_0^6+3q_0^5-14q_0^4-q_0^3+2q_0^2+q_0}{3},
\]
and
\begin{equation}\label{eq:Skabbounds}
\frac14-\frac1{4q_0}<\frac{a(X_q)}{g(X_q)}<\frac14.
\end{equation}
In particular, $a(X_q)/g(X_q)\to1/4$ as $s\to\infty$.
\end{maintheorem}

A second generalization of the GK curve was introduced by Beelen and Montanucci \cite{BMnew}. For the same $q$ and $n$, let $\mathcal B_{q,n}$ be the smooth projective model of
\begin{equation}\label{eq:BMmodel}
y^{q+1}=x^{q+1}+1,\qquad
z^{d_n}=y\frac{x^{q^2}+x}{x^{q+1}+1}.
\end{equation}
It is maximal over $\F_{q^{2n}}$ and has the same genus as $\mathcal G_{q,n}$. For $n=3$ it is isomorphic to the GK curve. For $n\ge5$ the two generalizations are nonisomorphic by \cite[Corollary~2.6]{BMnew}. We compute its $a$-number by applying the multipoint formula to the cover of the Hermitian curve in \eqref{eq:BMmodel}.

\begin{maintheorem}\label{thm:BM}
Let $q=2^h$, $h\ge1$, and let $n\ge3$ be odd. Then
\begin{equation}\label{eq:BManswer}
a(\mathcal B_{q,n})=
\frac{q^n(3q^3-6q^2+4q+2)+3q^4-10q^2+4q}{24(q-1)}.
\end{equation}
For $n=3$, this is the GK formula in Theorem~\ref{thm:GGS}. For each fixed odd $n\ge3$, the ratio $a(\mathcal B_{q,n})/g(\mathcal B_{q,n})$ tends to $1/4$ as $q=2^h$ tends to infinity.
\end{maintheorem}

\subsection*{Relation to previous work}
The identity $a(X)=\ell(\tfrac12\divis(dt))$ is classical. Friedlander, Garton, Malmskog, Pries, and Weir used its one-point form to compute the $a$-number of the Suzuki curve \cite[Section~3.3 and Theorem~4.1]{FGMPW}. We recall their calculation in Section~\ref{sec:Suzuki} and compute the further Riemann--Roch dimensions required for the Skabelund cover. The Hermitian value is known \cite{Gross,MS}, as is the value $q^2/16$ for the Abd\'on--Torres curve
\[
\sum_{i=0}^{h-1}y^{2^i}=x^{q+1},\qquad q=2^h\ge4,
\]
computed by Nourozi, Tafazolian, and Rahmati \cite[Theorem~4.1]{NTR}. Section~4 gives semigroup proofs of these formulas.

Both $\mathcal G_{q,n}$ and $\mathcal B_{q,n}$ admit cyclic models over $\mathbb P^1$. Their $a$-numbers therefore also follow from Elkin's characteristic-two formula \cite[proof of Corollary~1.4]{Elkin}. We give these models and their branch data in Section~\ref{subsec:cycliccomparison}. Theorems~\ref{thm:GGS} and~\ref{thm:BM} provide closed expressions for the resulting invariants, obtained here by Riemann--Roch calculations on the Hermitian base. Our transfer formulas apply to bases of arbitrary genus; the Skabelund calculation uses the Suzuki curve as base. The common-monomial fibre products in Section~\ref{sec:fibre} also admit cyclic models over $\mathbb P^1$.

The geometric data for the GK curve are given in \cite{GK,BM}. Skabelund stated the semigroup generators for his Suzuki cover \cite{Skabelund}, and Beelen, Landi, and Montanucci proved the semigroup equality \cite{BLM}. For the Beelen--Montanucci curve, the divisor of $dz$ and spaces of regular differentials were studied by Montanucci and Pallozzi Lavorante \cite{MPL}. We use these descriptions to evaluate the half-canonical dimensions. In Section~\ref{sec:fibre}, a genus-six fibre product gives a family with varying $a$-number and constant quadratic quotient genera and $a$-numbers. Its Ekedahl--Oort type is determined in Proposition~\ref{prop:fibreEO} by computing the required subspaces of de Rham cohomology.

\subsection*{Organization}
Section~2 proves the transfer formulas and relates them to Elkin's formula. Sections~3 and~4 collect the Suzuki and additive calculations. The GK and GGS curves are treated in Sections~5--7, the Beelen--Montanucci curves in Section~8, and the Skabelund covers in Section~9. Section~10 concerns Artin--Schreier fibre products and the genus-six Ekedahl--Oort calculation.

\section{The semigroup counting method}
All curves below are smooth, projective, and connected over $k=\overline{\F}_2$, unless an affine model is specified. For a divisor $E$ on a curve $Y$, write $\ell_Y(E)=\dim_k L_Y(E)$. The semigroup $H_Y(P)$ consists of the pole orders at $P$ of functions regular away from $P$, including zero. In particular,
\[
\ell_Y(jP)=\#\bigl(H_Y(P)\cap[0,j]\bigr)\quad(j\ge0).
\]
For a numerical semigroup $H$ and a positive element $e\in H$, its Ap\'ery set with respect to $e$ consists of the least element of $H$ in each residue class modulo $e$. If these elements are $w_0,\ldots,w_{e-1}$, then
\[
g(H)=\frac1e\sum_{i=0}^{e-1}w_i-\frac{e-1}{2}.
\]
Here $g(H)$ is the number of gaps.

We begin with the Cartier-kernel calculation used in \cite[Section~3.3, equation~(3.4), and proof of Theorem~4.1]{FGMPW}.

For every separating function $t$, the divisor of $dt$ has even coefficients in characteristic two. Locally, in a uniformizer $u$, write $t=\sum_j c_j u^j$. The nonzero terms of $dt$ have exponents $j-1$ with $j$ odd, so their least exponent is even. Thus $\tfrac12\divis(dt)$ is an integral divisor of degree $g-1$. If $s$ is another separating function, write $ds=(a^2+t b^2)dt$. Since $\Cart(ds)=0$, one has $b=0$; hence $\tfrac12\divis(ds)$ differs from $\tfrac12\divis(dt)$ by a principal divisor. This is the canonical theta characteristic of a characteristic-two curve; see, for example, \cite[Section 2]{RPT}. The next lemma expresses its number of sections as the $a$-number.

\begin{lemma}\label{lem:cartier}
Let $Y$ be a smooth projective curve over $k$, and let $t\in k(Y)$ be a separating variable. Suppose $\divis(dt)=2E$ for a divisor $E$. Then
\[
\ker\bigl(\Cart\vert_{\HH(Y,\Omega_Y^1)}\bigr)
=\{u^2dt:u\in L_Y(E)\},
\qquad a(Y)=\ell_Y(E).
\]
\end{lemma}
\begin{proof}
Since $k$ is perfect and $t$ is separating, every rational differential has a unique expression
\[
\omega=(u^2+t v^2)dt,\qquad u,v\in k(Y).
\]
The defining formula for Cartier gives $\Cart(\omega)=v\,dt$. Thus its rational kernel consists of the forms $u^2dt$. Such a form is regular if and only if
\[
\divis(u^2dt)=2\divis(u)+2E\ge0,
\]
equivalently $u\in L_Y(E)$. The map $u\mapsto u^2dt$ is additive and Frobenius-semilinear. Since Frobenius is an automorphism of $k$, its image has dimension $\ell_Y(E)$.
\end{proof}

\begin{proposition}[Semigroup counting principle]\label{prop:method}
Let $Y$ be a curve of genus $g\ge1$ in characteristic two. Suppose that a separating function $t$ satisfies
\[
\divis(dt)=(2g-2)P
\]
for a point $P\in Y(k)$. Then
\begin{equation}\label{eq:method}
a(Y)=\ell_Y((g-1)P)=\#\bigl(H_Y(P)\cap[0,g-1]\bigr).
\end{equation}
If $f_j$ is a function with pole divisor $jP$ for each nongap $0\le j\le g-1$, with $f_0=1$, then
\[
\{f_j^2dt:j\in H_Y(P),\ 0\le j\le g-1\}
\]
is a basis of the regular Cartier kernel.
\end{proposition}
\begin{proof}
Apply Lemma~\ref{lem:cartier} with $E=(g-1)P$. The functions $f_j$ have distinct pole orders and account for all the dimension jumps up to $g-1$, so they form a basis of $L_Y((g-1)P)$. Their squares times $dt$ give the asserted basis.
\end{proof}

The hypothesis concerns an exact differential. Symmetry of $H_Y(P)$ alone does not imply \eqref{eq:method}.

\begin{remark}\label{rem:char2}
The preceding reduction to one Riemann--Roch space is specific to characteristic two. In characteristic $p$, a separating variable gives
\[
\omega=\left(\sum_{i=0}^{p-1}u_i^pt^i\right)dt,
\qquad \Cart(\omega)=u_{p-1}dt.
\]
For $p>2$ its rational kernel has $p-1$ independent coefficients over $k(Y)^p$. Regularity of their sum need not imply regularity of the individual summands, so the same single-divisor count does not follow. For example, in characteristic three the smooth projective curve $y^3-y=x^4$ has genus $3$, with pole orders $3,4$ for $x,y$ and $\divis(dx)=4P_\infty$. The differential $y\,dx$ is regular and lies in the Cartier kernel, since
\[
y\,dx=(y^3-x^4)dx=\bigl(y^3+x(-x)^3\bigr)dx.
\]
Both summands $y^3dx$ and $-x^4dx$ have a pole of order $8$ at infinity, which cancels in their sum.
\end{remark}

\subsection{Transfer through odd-degree radical covers}
Let us first consider a radical cover with one point at infinity.

\begin{theorem}\label{thm:transfer}
Let $Y$ be a smooth projective curve over $k=\overline{\F}_2$, let $Q\in Y(k)$, and suppose $f\in k(Y)$ satisfies
\[
(f)_\infty=NQ,\qquad \divis_Y(df)=(2g_Y-2)Q.
\]
Here $N\ge1$. Let $m\ge1$ be odd with $\gcd(m,N)=1$, and let $X$ be the smooth projective curve with function field $k(Y)(z)$, where $z^m=f$. Put $H=H_Y(Q)$. Then there is a unique point $P$ above $Q$, and
\begin{align}
g_X&=m g_Y+\frac{(m-1)(N-1)}2,\label{eq:transfergenus}\\
\divis_X(dz)&=(2g_X-2)P,\label{eq:transferdiff}\\
H_X(P)&=mH+N\N.\label{eq:transferH}
\end{align}
In particular,
\begin{equation}\label{eq:transfera}
a(X)=\sum_{j=0}^{m-1}
\ell_Y\!\left(\left\lfloor\frac{g_X-1-jN}{m}\right\rfloor Q\right).
\end{equation}
If $\mathcal A=\operatorname{Ap}(H,N)$, the same number is
\begin{equation}\label{eq:transferApery}
a(X)=\sum_{w\in\mathcal A}
\max\left\{0,\,1+\left\lfloor\frac{g_X-1-mw}{N}\right\rfloor\right\}.
\end{equation}
\end{theorem}
\begin{proof}
The differential $df$ is a unit at every point of $Y\setminus\{Q\}$. Consequently, the zeros of $f$ are simple, and there are $N$ of them. These zeros also show that $f$ is not a proper power. The extension has degree $m$, is totally ramified at the $N$ zeros and at $Q$, and is unramified elsewhere. The assertion also covers $m=1$ in the evident way. Riemann--Hurwitz gives \eqref{eq:transfergenus}.

Differentiation gives $dz=z^{1-m}df$. At each finite zero of $f$, the tame different contributes $m-1$ to the pulled-back differential divisor, and $z^{1-m}$ contributes $1-m$. The contributions cancel. Away from these zeros and $P$, both factors are units. Thus $dz$ is supported at $P$, and its degree proves \eqref{eq:transferdiff}.

Pullbacks of functions on $Y$ and the function $z$ show that $H'=mH+N\N$ is contained in $H_X(P)$. Since $N\in H$, every element of $H$ has a unique form $w+cN$, with $w\in\mathcal A$ and $c\ge0$. Since $\gcd(m,N)=1$, the least representatives for $H'$ modulo $N$ are precisely the elements $mw$, $w\in\mathcal A$. Hence
\[
g(H')=\frac{m}{N}\sum_{w\in\mathcal A}w-\frac{N-1}{2}
=m g_Y+\frac{(m-1)(N-1)}2=g_X.
\]
Equality of the genera proves \eqref{eq:transferH}.

Every element of $H'$ also has a unique representation $mv+jN$, with $v\in H$ and $0\le j<m$: reduce $j$ modulo $m$ using $N\in H$, and use $\gcd(m,N)=1$ for uniqueness. Counting the elements up to $g_X-1$ proves \eqref{eq:transfera}. Counting instead by the Ap\'ery representatives $mw$ proves \eqref{eq:transferApery}. For $g_X\ge1$, Proposition~\ref{prop:method} identifies both counts with $a(X)$. For $g_X=0$, both counts are zero, as required.
\end{proof}

\begin{corollary}\label{cor:periodicity}
Fix $Y,Q,f,N$ as in Theorem~\ref{thm:transfer}. For positive odd $m$ coprime to $N$, denote the corresponding $a$-number by $a_m$. Put
\[
A=2g_Y+N-1,\qquad
\mathcal W=\{w\in\operatorname{Ap}(H_Y(Q),N):2w<A\},
\qquad S=\sum_{w\in\mathcal W}(A-2w).
\]
Then
\begin{equation}\label{eq:generalperiod}
a_m=\sum_{w\in\mathcal W}\left\lfloor\frac{m(A-2w)+N-1}{2N}\right\rfloor,
\qquad a_{m+2N}=a_m+S.
\end{equation}
Thus the $a$-number is an affine function of $m$ on each admissible residue class modulo $2N$. If $A>0$, then
\[
\lim_{\substack{m\to\infty\\m\text{ odd},\ (m,N)=1}}
\frac{a_m}{g_X}=\frac{S}{NA}.
\]
\end{corollary}
\begin{proof}
Substitute $g_X-1=(mA-N-1)/2$ in \eqref{eq:transferApery}. The summand is zero if $A-2w\le0$ and otherwise equals the displayed floor. Increasing $m$ by $2N$ adds $A-2w$ to it. Finally, $a_m=mS/(2N)+O(1)$ and $g_X=mA/2+O(1)$.
\end{proof}

For fixed $H$, $N$, and $m$ satisfying Theorem~\ref{thm:transfer}, the $a$-number is independent of the remaining coefficients. The construction can also be iterated. For example, $f_1=z^{2^v}+z$, with $v\ge1$, satisfies $df_1=dz$, has pole order $2^vN$ at $P$, and is regular elsewhere. A further odd radical degree coprime to $2^vN$ is therefore permitted.

\subsection{A multipoint transfer formula}\label{subsec:multi}
For the Beelen--Montanucci curves we need the following multipoint version.

\begin{theorem}\label{thm:multitransfer}
Let $Y$ be a smooth projective curve in characteristic two, and let $f$ be a nonconstant function with
\[
(f)_\infty=\sum_{i=1}^{s}N_iQ_i,\qquad
\divis_Y(df)=\sum_{i=1}^{s}c_iQ_i,
\]
where the $Q_i$ are distinct, $N_i\ge1$, and the $c_i$ are even integers. Let $m\ge1$ be odd and coprime to every $N_i$, and let $X$ be the smooth projective curve with function field $k(Y)(z)$, where $z^m=f$. There is a unique point $P_i$ above each $Q_i$. Put
\[
b_i=\frac{m c_i+(m-1)(N_i+1)}2.
\]
Then
\begin{equation}\label{eq:multidz}
\divis_X(dz)=2\sum_{i=1}^{s}b_iP_i,
\end{equation}
and
\begin{equation}\label{eq:multitransfer}
a(X)=\sum_{j=0}^{m-1}
\ell_Y\!\left(\sum_{i=1}^{s}
\left\lfloor\frac{b_i-jN_i}{m}\right\rfloor Q_i\right).
\end{equation}
\end{theorem}
\begin{proof}
Since $df$ is a unit outside the $Q_i$, all zeros of $f$ are simple. In particular $f$ is not a proper power, so the extension has degree $m$. It is totally and tamely ramified at these zeros and at the $Q_i$, and is unramified elsewhere. At a zero of $f$, the identity $dz=z^{1-m}df$ cancels the tame different contribution. At $P_i$, its order is
\[
m c_i+(m-1)+(m-1)N_i=2b_i,
\]
which proves \eqref{eq:multidz}. Lemma~\ref{lem:cartier} gives $a(X)=\ell_X(\sum b_iP_i)$.

This Riemann--Roch space is invariant under the cyclic group of the cover. Since $m$ is odd, its character projectors are defined, and it is the direct sum of its character spaces. A rational function in the $j$th character space is $h z^j$, where $h\in k(Y)$ and $0\le j<m$. At a point above a simple zero of $f$, its valuation is $m v(h)+j$, so regularity is equivalent to $v(h)\ge0$. The same requirement holds at all other points outside the $Q_i$. At $P_i$ the required inequality is
\[
m v_{Q_i}(h)-jN_i\ge-b_i,
\]
or $v_{Q_i}(h)\ge-\lfloor(b_i-jN_i)/m\rfloor$. Thus the $j$th character space has the dimension appearing in \eqref{eq:multitransfer}. Summing proves the formula. For $m=1$ the same argument reduces directly to Lemma~\ref{lem:cartier}.
\end{proof}

\subsection{Preservation under covers unramified away from one point}
\begin{proposition}\label{prop:wildpreserve}
Suppose $\divis_Y(dt)=(2g_Y-2)Q$. Let $\pi:X\to Y$ be a finite separable morphism unramified over $Y\setminus\{Q\}$, with a unique point $P$ above $Q$. Then
\[
\divis_X(d\pi^*t)=(2g_X-2)P,
\qquad a(X)=\#\bigl(H_X(P)\cap[0,g_X-1]\bigr).
\]
\end{proposition}
\begin{proof}
The differential divisor on $X$ is the pullback of the differential divisor on $Y$ plus the different divisor. Both are supported at $P$. Their total degree is $2g_X-2$, proving the first assertion. The pullback differential is nonzero because $\pi$ is separable. For $g_X\ge1$, Proposition~\ref{prop:method} gives the second assertion; the genus-zero case is immediate.
\end{proof}
The cover may be wildly ramified. Its pole semigroup is a separate input.

\subsection{An effective counting formula}
For any numerical semigroup $H$ and positive $e\in H$, let $w_0,\ldots,w_{e-1}$ be its Ap\'ery representatives modulo $e$. For every integer $B$,
\begin{equation}\label{eq:algorithm}
\#\{h\in H:h\le B\}
=\sum_{i=0}^{e-1}\max\left\{0,\,1+\left\lfloor\frac{B-w_i}{e}\right\rfloor\right\}.
\end{equation}
Indeed, each element is uniquely $w_i+ce$, with $c\ge0$. Under the hypothesis of Proposition~\ref{prop:method}, this gives the $a$-number directly from the Ap\'ery set.

\subsection{Comparison with cyclic covers of the projective line}\label{subsec:cycliccomparison}
For a connected tame cyclic cover of $\mathbb P^1$ of odd degree $K>1$, let $e_1,\ldots,e_s$ be its nonzero Kummer exponents modulo $K$, including the exponent at infinity if it is ramified. Set
\[
d_i=-1+\sum_{j=1}^s\left\{\frac{i e_j}{K}\right\},
\qquad \sigma(i)\equiv2^{-1}i\pmod K,
\qquad 1\le i,\sigma(i)<K,
\]
where braces denote fractional part. These are the character-space dimensions, with $g=\sum_{i=1}^{K-1}d_i$. Elkin's characteristic-two argument \cite[proof of Corollary~1.4]{Elkin} gives
\begin{equation}\label{eq:Elkinrank}
\operatorname{rank}\Cart=\sum_{i=1}^{K-1}\min\{d_i,d_{\sigma(i)}\},
\qquad
a=g-\sum_{i=1}^{K-1}\min\{d_i,d_{\sigma(i)}\}.
\end{equation}

For the GGS model \eqref{eq:GGSmodel}, put $m=d_n$, $K=m(q+1)=q^n+1$, and
\[
B(x)=x^q+x,\qquad H(x)=B(x)^{q-1}+1.
\]
Since $y^{q^2}+y=yH(x)$, elimination of $y$ gives
\begin{equation}\label{eq:GGSKummer}
z^K=B(x)H(x)^{q+1},\qquad y=\frac{z^m}{H(x)}.
\end{equation}
Thus this equation has the full function field of $\mathcal G_{q,n}$. The polynomial $B$ has $q$ simple roots. At a root of $H$, one has $B\ne0$ and $H'=B^{q-2}\ne0$, so $H$ has $q^2-q$ simple roots disjoint from those of $B$. The right side of \eqref{eq:GGSKummer} has degree $q^3$. Consequently its Kummer exponents and their multiplicities are
\begin{equation}\label{eq:GGSbranchdata}
(1,q),\qquad(q+1,q^2-q),\qquad(-q^3\bmod K,1).
\end{equation}
The simple zeros with exponent one show that the extension has degree $K$. The same elimination applies to $C_{q,d}$ in Section~5, with $K=d(q+1)$; any exponents congruent to zero are then omitted from the branch list.

For the Beelen--Montanucci model \eqref{eq:BMmodel}, elimination gives instead
\begin{equation}\label{eq:BMKummer}
z^K=\frac{(x^{q^2}+x)^{q+1}}{(x^{q+1}+1)^q},
\qquad y=\frac{z^m(x^{q+1}+1)}{x^{q^2}+x}.
\end{equation}
Every root of $x^{q+1}+1$ is a root of $x^{q^2}+x$: from $x^{q+1}=1$ one gets $x^{q^2-1}=1$. Both polynomials have simple roots. At their $q+1$ common roots the net exponent is one; at the other $q^2-q-1$ roots of the numerator it is $q+1$. The pole order of the right side at infinity is $q^3-q$. Hence the branch data are
\begin{equation}\label{eq:BMbranchdata}
(1,q+1),\qquad(q+1,q^2-q-1),\qquad(-(q^3-q)\bmod K,1).
\end{equation}
As before, the exponent one guarantees degree $K$, and the expression for $y$ proves equality of function fields.

Substituting these branch data in \eqref{eq:Elkinrank} gives finite sums for the invariants in Theorems~\ref{thm:GGS} and~\ref{thm:BM}. Below we evaluate the same invariants using Riemann--Roch spaces on the Hermitian curve.

\section{The Suzuki curve}\label{sec:Suzuki}
We record the semigroup calculation of Friedlander et al.\ \cite[Theorem 4.1]{FGMPW}, retaining the small-degree dimensions needed for the Skabelund cover.

Let $r=2^s$, $q=2r^2$, and ${\Delta_S}=q+2r+1$. The Suzuki curve is
\[
S_q:\quad y^q+y=x^r(x^q+x).
\]
It has a unique point $Q_\infty$ at infinity, and the required geometric data are
\begin{align}
g_S:=g(S_q)&=r(q-1), &
\divis_{S_q}(dx)&=(2g_S-2)Q_\infty,\label{eq:suzukican}\\
H_S:=H_{S_q}(Q_\infty)&=\langle q,q+r,q+2r,{\Delta_S}\rangle.\label{eq:suzukisg}
\end{align}
These facts are recalled in \cite[Lemmas~3.4 and 3.5 and Proposition~3.6]{FGMPW}; the semigroup is due to Hansen and Stichtenoth \cite[Proposition~1.5]{HS}. In \cite{FGMPW} the curve is written $z^q+z=y^{r}(y^q+y)$, so their functions $y,z,h_1,h_2$ correspond to our $x,y,z,w$ in Section~\ref{sec:Skabexamples}. The canonical divisor in \eqref{eq:suzukican} also follows from Proposition~\ref{prop:wildpreserve}, applied to $x:S_q\to\mathbb P^1$ and $\divis_{\mathbb P^1}(dx)=-2\infty$: the affine model is \'etale over the $x$-line, since the derivative of $y^q+y$ with respect to $y$ is $1$. Set
\[
\mathcal L_S(n)=\ell_{S_q}(n\Delta_S Q_\infty)\quad(n\in\mathbb Z),
\]
where $\mathcal L_S(n)=0$ for $n<0$. The next normal form follows from the semigroup relations; its use in the $a$-number count is the same as in \cite[Proposition 3.7 and proof of Theorem 4.1]{FGMPW}.

\begin{lemma}\label{lem:normal}
Every element of $H_S$ has a unique expression
\begin{equation}\label{eq:normal}
\alpha q+\beta(q+r)+\gamma(q+2r)+\delta {\Delta_S},
\qquad \alpha\ge0,\quad 0\le \beta\le1,\quad 0\le \gamma,\delta<r.
\end{equation}
\end{lemma}
\begin{proof}
Existence follows from the numerical relations
\[
r\Delta_S=rq+(q+r),\qquad
2(q+r)=q+(q+2r),\qquad
r(q+2r)=(r+1)q.
\]
First reduce the coefficient of ${\Delta_S}$ modulo $r$, then that of $q+r$ modulo $2$, and finally that of $q+2r$ modulo $r$. All new coefficients remain nonnegative.

For uniqueness, suppose two expressions in \eqref{eq:normal} agree. Reduction modulo $r$ gives $\delta=\delta'$, since ${\Delta_S}\equiv1\pmod r$. After canceling these terms, reduction modulo $2r$ gives $\beta=\beta'$. After canceling again, reduction modulo $q=2r^2$ gives $2r(\gamma-\gamma')\equiv0\pmod q$, so $\gamma=\gamma'$. Finally $\alpha=\alpha'$.
\end{proof}

\begin{lemma}\label{lem:small}
For $0\le n\le r-1$,
\begin{equation}\label{eq:small}
\mathcal L_S(n)=\binom{n+3}{3}+\binom{n+2}{3}
=\frac{(n+1)(n+2)(2n+3)}6.
\end{equation}
\end{lemma}
\begin{proof}
In this range,
\[
n({\Delta_S}-q)=n(2r+1)\le(r-1)(2r+1)<q.
\]
Thus $(n+1)q>n\Delta_S$. If an expression in \eqref{eq:normal} has $\alpha+\beta+\gamma+\delta\ge n+1$, its value exceeds $n\Delta_S$, since all four generators are at least $q$. Conversely, if $\alpha+\beta+\gamma+\delta\le n$, its value is at most $n\Delta_S$, since all four generators are at most ${\Delta_S}$. Therefore $\mathcal L_S(n)$ is the number of tuples with
\[
\alpha+\beta+\gamma+\delta\le n,\qquad \alpha,\gamma,\delta\ge0,\qquad \beta\in\{0,1\}.
\]
The bounds $\gamma,\delta<r$ are automatic. For $\beta=0$ the number is $\binom{n+3}{3}$, and for $\beta=1$ it is $\binom{n+2}{3}$, with the latter understood as zero when $n=0$. Lemma~\ref{lem:normal} ensures that no pole order is counted twice.
\end{proof}

\begin{proposition}[Friedlander--Garton--Malmskog--Pries--Weir]\label{prop:SuzukiKnown}
For $q=2r^2$ with $r=2^s$ and $s\ge1$,
\[
a(S_q)=\frac{r(r+1)(2r+1)}6.
\]
\end{proposition}
\begin{proof}
We have $g_S-1=(r-1){\Delta_S}$. Proposition~\ref{prop:method} applied to $dx$ and Lemma~\ref{lem:small} give
\[
a(S_q)=\ell_{S_q}((g_S-1)Q_\infty)=\mathcal L_S(r-1)
=\binom{r+2}{3}+\binom{r+1}{3}
=\frac{r(r+1)(2r+1)}6.
\]
\end{proof}

\section{Two-generator semigroups and additive families}
We shall also use the following calculation for additive covers.

\begin{proposition}\label{prop:additive}
Let $q=2^h$, $h\ge1$, and let $F(x)\in k[x]$ have odd degree $M\ge1$. The smooth projective curve
\[
Y_F:\quad y^q+y=F(x)
\]
has a unique point $P$ at infinity, genus $g=(q-1)(M-1)/2$, and semigroup $H_Y(P)=\langle q,M\rangle$. Moreover,
\begin{equation}\label{eq:additivea}
a(Y_F)=\sum_{j=0}^{q/2-1}
\left\lfloor\frac{(2j+1)M+q-1}{2q}\right\rfloor.
\end{equation}
In particular, the $a$-number depends only on $q$ and $M$.
\end{proposition}
\begin{proof}
If $M=1$, the curve is rational and both sides of \eqref{eq:additivea} are zero. Assume henceforth that $M\ge3$. The polynomial is irreducible over $k(x)$. Indeed, for a valuation at infinity in an extension generated by a root, let $e$ be the ramification index and let the root have pole order $b$. Comparing pole orders gives $qb=eM$. Since $\gcd(q,M)=1$, one has $q\mid e$, so the extension has degree $q$, is totally ramified at infinity, and has only one point above infinity. The pole orders of $x,y$ are $q,M$.

The cover of the $x$-line is Galois with additive group $\F_q$. Each nontrivial quadratic character gives an Artin--Schreier quotient with right side $cF(x)$ for some $c\ne0$, of odd pole order $M$. Its conductor is $M+1$. The conductor--different formula gives different exponent $(q-1)(M+1)$ at infinity. Riemann--Hurwitz therefore gives the asserted genus. The semigroup $\langle q,M\rangle$ is contained in $H_Y(P)$ and has the same genus, so they are equal. The affine projection to the $x$-line is \'{e}tale, since the derivative of the defining polynomial with respect to $y$ is $1$. Hence $\divis(dx)=(2g-2)P$, as also follows from Proposition~\ref{prop:wildpreserve}.

The Ap\'ery representatives modulo $q$ are $bM$, $0\le b<q$. By \eqref{eq:algorithm},
\[
a(Y_F)=\sum_{b=0}^{q-1}
\max\left\{0,\left\lfloor\frac{(q-1-2b)M+q-1}{2q}\right\rfloor\right\}.
\]
For $b\ge q/2$ the floor is nonpositive, and for $b<q/2$ it is nonnegative. Reversing the order of the latter values proves \eqref{eq:additivea}.
\end{proof}

\begin{corollary}\label{cor:additiveprogressions}
If $M=q\lambda+\varepsilon$, where $\lambda\ge1$ and $\varepsilon\in\{-1,1\}$, then
\[
a(Y_F)=\frac{q^2\lambda}{8}+
\begin{cases}
0,&\lambda\text{ even},\\
\varepsilon q/4,&\lambda\text{ odd}.
\end{cases}
\]
More generally, replacing $M$ by $M+2q$ increases the $a$-number by $q^2/4$.
\end{corollary}
\begin{proof}
Substitute $M=q\lambda+\varepsilon$ in \eqref{eq:additivea} and distinguish even and odd $\lambda$. If $\lambda=2u$, each floor is $(2j+1)u$. If $\lambda=2u+1$, it is $(2j+1)u+j+1$ for $\varepsilon=1$, and $(2j+1)u+j$ for $\varepsilon=-1$. Summing proves the first assertion. The second follows because the change in the $j$th floor is $2j+1$.
\end{proof}

\begin{remark}\label{rem:additivechecks}
For $q=2$, \eqref{eq:additivea} reduces to $a(Y_F)=\lfloor(M+1)/4\rfloor$; for $M\equiv1\pmod4$ this is $(M-1)/4$, as in \cite[Remark~3.2]{FGMPW}. More generally, Elkin and Pries \cite{EP} showed that the $2$-torsion group scheme of a hyperelliptic curve in characteristic two depends only on its ramification invariants. For $F(x)=x^{q+1}$, the curve $Y_F$ is the Hermitian curve $\mathcal H_q$, maximal over $\F_{q^2}$ (with $x$ and $y$ interchanged relative to \eqref{eq:cyclicHerm} below), and Corollary~\ref{cor:additiveprogressions} with $\lambda=\varepsilon=1$ recovers the known formula \cite{Gross,MS}
\begin{equation}\label{eq:Hermitian-a}
a(\mathcal H_q)=\frac{q(q+2)}8.
\end{equation}
\end{remark}

For fixed $q$ and odd $M\to\infty$, the floor formula gives
\[
\lim\frac{a(Y_F)}{g(Y_F)}=\frac{q}{4(q-1)}.
\]
The same method applies to smooth affine plane models with constant nonzero derivative in $y$ and a unique point at infinity, once their pole semigroups are known.

\subsection{The Abd\'on--Torres curve}
\begin{proposition}[Nourozi--Tafazolian--Rahmati \cite{NTR}]\label{prop:ATknown}
Let $q=2^h\ge4$, and let $\mathcal A_q$ be the smooth projective model of
\[
y^{q/2}+y^{q/4}+\cdots+y=x^{q+1}.
\]
Then $a(\mathcal A_q)=q^2/16$.
\end{proposition}
\begin{proof}
By Abd\'on and Torres \cite{AT}, this curve is maximal over $\F_{q^2}$, has genus $g=q(q-2)/4$, and has a unique point $P$ at infinity. The pole orders of $x,y$ are $q/2,q+1$. The semigroup they generate has genus $g$, hence
\[
H(P)=\langle q/2,q+1\rangle.
\]
The affine equation is \'etale over the $x$-line because its derivative in $y$ is $1$. Therefore $\divis(dx)=(2g-2)P$. Put $R=q/2$. Each nongap has a unique expression $iR+j(2R+1)$ with $i\ge0$ and $0\le j<R$. Since $g-1=R(R-1)-1$, the terms at most $g-1$ have $0\le j<R/2$, and for each such $j$ there are $R-1-2j$ choices of $i$. Consequently
\[
a(\mathcal A_q)=\sum_{j=0}^{R/2-1}(R-1-2j)=\frac{R^2}{4}=\frac{q^2}{16}.
\]
For $q=2$ the curve is rational and its $a$-number is zero.
\end{proof}

\section{Cyclic covers of the Hermitian curve}
Let $q=2^h$, $h\ge1$, and let $d\ge1$ be odd. We consider the smooth projective model $C_{q,d}$ of
\begin{equation}\label{eq:cyclicHerm}
x^q+x=y^{q+1},\qquad z^d=y^{q^2}+y.
\end{equation}
Set
\[
N=q^3,\qquad A=q^3+q^2-q-1=(q+1)(q^2-1).
\]
The GK and GGS curves occur at the degrees considered in Sections~6 and~7.

\begin{proposition}\label{prop:Hermgeometry}
The curve $C_{q,d}$ has a unique point $P_\infty$ above the point $Q_\infty$ at infinity of the Hermitian curve. Its genus and a canonical differential satisfy
\begin{equation}\label{eq:Hermgeometry}
g_{q,d}=\frac{dA-N+1}{2},\qquad
\divis(dz)=(2g_{q,d}-2)P_\infty.
\end{equation}
Moreover,
\begin{equation}\label{eq:Hermsemigroup}
H_{C_{q,d}}(P_\infty)=\langle dq,d(q+1),N\rangle,
\end{equation}
and its Ap\'ery set with respect to $N$ is
\begin{equation}\label{eq:HermApery}
\left\{d\bigl(aq+b(q+1)\bigr):0\le a<q^2,\ 0\le b<q\right\}.
\end{equation}
\end{proposition}
\begin{proof}
On the Hermitian curve the pole orders of $y,x$ are $q,q+1$, and $\divis(dy)=(q^2-q-2)Q_\infty$; the latter is Proposition~\ref{prop:wildpreserve} applied to $y:\mathcal H_q\to\mathbb P^1$, since the affine model is \'etale over the $y$-line. Set $f=y^{q^2}+y$. Then $(f)_\infty=NQ_\infty$ and $df=dy$. Theorem~\ref{thm:transfer}, with base semigroup $\langle q,q+1\rangle$, gives \eqref{eq:Hermgeometry} and \eqref{eq:Hermsemigroup}.

Every element of this semigroup has a unique form
\[
d\bigl(aq+b(q+1)\bigr)+cN,
\qquad 0\le a<q^2,\quad 0\le b<q,\quad c\ge0.
\]
Existence follows by first reducing $b$ modulo $q$ and then $a$ modulo $q^2$. For uniqueness, reduce modulo $N$, cancel the unit $d$, and reduce successively modulo $q$ and $q^3$. This determines $b$ and $a$, and then $c$. Taking $c=0$ proves \eqref{eq:HermApery}.
\end{proof}

\begin{corollary}\label{cor:Hermfloors}
For odd $d\ge1$,
\begin{equation}\label{eq:floorcount}
a(C_{q,d})=\sum_{\substack{0\le a<q^2,\ 0\le b<q\\ T_{a,b}>0}}
\left\lfloor\frac{dT_{a,b}+N-1}{2N}\right\rfloor,
\qquad T_{a,b}=A-2aq-2b(q+1).
\end{equation}
A basis of the Cartier kernel consists of the forms
\[
y^{2a}x^{2b}z^{2c}\,dz,
\quad 0\le a<q^2,\ 0\le b<q,\ c\ge0,
\quad d\bigl(aq+b(q+1)\bigr)+cN\le g_{q,d}-1.
\]
\end{corollary}
\begin{proof}
Proposition~\ref{prop:Hermgeometry} and Lemma~\ref{lem:cartier} reduce the calculation to nongaps at most $g_{q,d}-1$. For an Ap\'ery element $dv$, the number of allowed multiples of $N$ is
\[
\max\left\{0,\left\lfloor\frac{g_{q,d}-1-dv}{N}\right\rfloor+1\right\}
=\max\left\{0,\left\lfloor\frac{d(A-2v)+N-1}{2N}\right\rfloor\right\}.
\]
The integer $A-2v$ is odd. If it is negative, the displayed floor is nonpositive; if it is positive, the floor is nonnegative. This proves \eqref{eq:floorcount}. The functions $y^ax^bz^c$ in the specified range have distinct pole orders and exhaust the relevant nongaps. Squaring them and multiplying by $dz$ gives the asserted basis.
\end{proof}

\section{The Giulietti--Korchm\'aros curve}
We first evaluate the count for $d=d_0:=q^2-q+1$, so that $C_{q,d_0}$ is the GK curve. This curve was introduced in \cite{GK}; its semigroups and the divisor of $dz$ are also discussed in \cite{BM}.

\begin{proposition}\label{prop:GK}
For $q=2^h$, $h\ge1$,
\begin{equation}\label{eq:GKanswer}
a(C_{q,d_0})=\frac{q(q+1)(3q^3-6q^2+10q-4)}{24}.
\end{equation}
\end{proposition}
\begin{proof}
Write $H_{\mathcal H}=\langle q,q+1\rangle$ for the Hermitian semigroup at infinity and set
\[
L_{\mathcal H}(n)=\ell_{\mathcal H_q}(n(q+1)Q_\infty),\qquad r=q/2.
\]
Every element of $d_0H_{\mathcal H}+N\N$ has a unique expression $d_0v+jN$, with $v\in H_{\mathcal H}$ and $0\le j<d_0$. Indeed $N\in H_{\mathcal H}$ and $\gcd(d_0,N)=1$. The identities
\[
d_0(q+1)=N+1,\qquad
g_{q,d_0}-1=(N+1)(q^2/2-1)
\]
therefore give
\begin{equation}\label{eq:GKsum}
a(C_{q,d_0})=\sum_{n=0}^{q^2/2-1}L_{\mathcal H}(n).
\end{equation}
Here the reduction can also be checked directly: the bound on $v$ is
\[
\left\lfloor\frac{g_{q,d_0}-1-jN}{d_0}\right\rfloor
=(q+1)(q^2/2-1-j),
\]
and the negative arguments contribute zero.

For $0\le n\le r-1$, the inequality $(n+1)q>n(q+1)$ and the normal forms $aq+b(q+1)$, $a\ge0$, $0\le b<q$, show that
\[
L_{\mathcal H}(n)=\binom{n+2}{2}.
\]
Since $g_{\mathcal H}-1=(r-1)(q+1)$ and $\divis(dy)=(2g_{\mathcal H}-2)Q_\infty$, Riemann--Roch gives
\[
L_{\mathcal H}(n)=(q+1)(n-r+1)+L_{\mathcal H}(2r-2-n).
\]
Splitting \eqref{eq:GKsum} at $r-1$, with $J=q(q-1)/2$, yields
\begin{align*}
a(C_{q,d_0})
&=\frac{(q+1)J(J+1)}2
+2\sum_{n=0}^{r-2}\binom{n+2}{2}+\binom{r+1}{2}\\
&=\frac{(q+1)J(J+1)}2+\frac{r(r+1)(2r+1)}6.
\end{align*}
Expansion proves \eqref{eq:GKanswer}. As in the Suzuki case, the middle term $L_{\mathcal H}(r-1)=\binom{r+1}{2}=q(q+2)/8$ is the $a$-number \eqref{eq:Hermitian-a} of the base curve.
\end{proof}

\section{The Garcia--G\"uneri--Stichtenoth curves}
Let $n\ge3$ be odd and set $d_n=(q^n+1)/(q+1)$. Denote $C_{q,d_n}$ by $\mathcal G_{q,n}$. These are the GGS curves \cite{GGS}; their unique point at infinity and genus are recalled in \cite[Section 2.1]{GMP}. We have
\[
g(\mathcal G_{q,n})=\frac{(q-1)(q^{n+1}+q^n-q^2)}2.
\]

\begin{lemma}\label{lem:arithmetic}
Put $d_0=q^2-q+1$ and
\[
\mathcal P=\{(a,b):0\le a<q^2,\ 0\le b<q,\ T_{a,b}>0\},
\qquad T_{a,b}=A-2aq-2b(q+1).
\]
For each $(a,b)\in\mathcal P$, let $R_{a,b}$ be the least nonnegative residue of $d_0T_{a,b}+N-1$ modulo $2N$. Then
\begin{align}
|\mathcal P|&=\frac N2,\label{eq:M}\\
\Sigma_q:=\sum_{(a,b)\in\mathcal P}T_{a,b}
&=\frac{3q^6+q^4+2q^3-q^2-2q}{12},\label{eq:S}\\
E_q:=\#\{(a,b)\in\mathcal P:R_{a,b}\ge N\}
&=\frac{q(2q^2-3q+2)}8.\label{eq:E}
\end{align}
\end{lemma}
\begin{proof}
Write $q=2r$, and put $\epsilon_b=0$ for $0\le b<r$ and $\epsilon_b=1$ for $r\le b<q$. For a fixed $b$, the positive values of $T_{a,b}$ occur precisely for
\begin{equation}\label{eq:cb}
0\le a<c_b,\qquad c_b=q^2/2+q/2-b-\epsilon_b.
\end{equation}
In particular, $\sum_b c_b=q^3/2$, proving \eqref{eq:M}. Summing the arithmetic progressions in $a$ gives
\[
\Sigma_q=\sum_{b=0}^{q-1}
c_b\bigl(A-2b(q+1)-q(c_b-1)\bigr).
\]
To evaluate this explicitly, pair $b$ with $q-1-b$ for $0\le b<r$. If $C=q^2/2+r$, their respective values of $c_b$ are $C-b$ and $C-q+b$. Substitution in the preceding sum, followed by
\[
\sum_{b=0}^{r-1}b=\frac{r(r-1)}2,\qquad
\sum_{b=0}^{r-1}b^2=\frac{r(r-1)(2r-1)}6,
\]
gives
\[
\Sigma_q=\frac{48r^6+4r^4+4r^3-r^2-r}{3},
\]
which is \eqref{eq:S}.

For \eqref{eq:E}, the numerator defining $R_{a,b}$ is even. Using $d_0(q+1)=N+1$, reduction modulo $N$ after division by $2$ gives
\[
\frac{d_0T_{a,b}+N-1}{2}
\equiv q^2/2-1+a(q^2-q)-b\pmod N.
\]
The right side is
\[
q\bigl(r-1+a(q-1)\bigr)+(q-1-b).
\]
As $0\le q-1-b<q$, the condition $R_{a,b}\ge N$ is therefore equivalent to
\begin{equation}\label{eq:Econdition}
\bigl(r-1+a(q-1)\bigr)\bmod q^2\ \ge q^2/2.
\end{equation}
This condition depends only on $a$.

Write $a=qu+v$, $0\le v<q$. The integer in \eqref{eq:Econdition} is congruent modulo $q^2$ to
\[
q(v-u)+r-1-v.
\]
For $v<r$ the final term is nonnegative, whereas for $v\ge r$ one rewrites this expression as
\[
q(v-u-1)+(q+r-1-v),\qquad 0\le q+r-1-v<q.
\]
Thus its leading base-$q$ digit is respectively $(v-u)\bmod q$ or $(v-u-1)\bmod q$. For $0\le u<r$, condition \eqref{eq:Econdition} holds exactly when
\[
\begin{cases}
v<u,&0\le v<r,\\
v\ge r+u+1,&r\le v<q.
\end{cases}
\]
Each such block contains $r-1$ values. All these values have $a<q^2/2-r$, so each is paired with every $b\in\{0,\ldots,q-1\}$ in \eqref{eq:cb}. There are no further solutions with $a<q^2/2$. The remaining possible values are $a=q^2/2+v$, $0\le v<r$. All satisfy \eqref{eq:Econdition}, and \eqref{eq:cb} permits exactly $r-v$ choices of $b$. Hence
\[
E_q=qr(r-1)+\sum_{v=0}^{r-1}(r-v)
=qr(r-1)+\frac{r(r+1)}2,
\]
which is \eqref{eq:E}.
\end{proof}

\begin{proof}[Proof of Theorem~\ref{thm:GGS}]
The case $n=3$ is Proposition~\ref{prop:GK}. Suppose $n\ge5$. Set
\[
u_n=\frac{q^{n-3}-q-2}{2(q+1)}.
\]
This is a nonnegative integer: $n-3$ is even, so the numerator is divisible by $q+1$, and the quotient is even because $q+1$ is odd. Direct calculation gives
\begin{equation}\label{eq:dnshift}
d_n=d_0+N+2Nu_n.
\end{equation}
For an odd positive $T$, write $d_0T+N-1=2NF+R$, $0\le R<2N$. Then
\[
\left\lfloor\frac{(d_0+N)T+N-1}{2N}\right\rfloor
=F+\frac{T-1}{2}+\boldsymbol{1}_{R\ge N}.
\]
Summing over $\mathcal P$ and applying Corollary~\ref{cor:Hermfloors} and Lemma~\ref{lem:arithmetic}, we obtain
\begin{equation}\label{eq:ashift}
a(C_{q,d_0+N})=a(C_{q,d_0})+\frac{\Sigma_q-N/2}{2}+E_q.
\end{equation}
Increasing $d$ by $2Nu$ adds $uT$ to each summand in \eqref{eq:floorcount}. Equations \eqref{eq:dnshift} and \eqref{eq:ashift} therefore yield
\[
a(\mathcal G_{q,n})
=u_n\Sigma_q+a(C_{q,d_0})+\frac{\Sigma_q-N/2}{2}+E_q.
\]
Substituting \eqref{eq:GKanswer}, \eqref{eq:S}, and \eqref{eq:E} first gives
\[
a(\mathcal G_{q,n})
=u_n\Sigma_q+\frac{3q^6+3q^5-2q^4+6q^3-4q^2}{24}.
\]
Inserting $u_n$ and collecting terms proves \eqref{eq:GGSanswer}.
\end{proof}

\begin{remark}
The restriction $n\ge5$ in \eqref{eq:GGSanswer} is necessary. Substituting $n=3$ into its right side gives the GK $a$-number minus $q(3q-2)/8$. The separate formula for $n=3$ reflects the different residue of $d_3$ modulo $2q^3$.
\end{remark}

For example,
\[
a(\mathcal G_{2,n})=3\cdot2^{n-3}\quad(n\ge5\text{ odd}),
\qquad a(\mathcal G_{2,3})=4.
\]
For each fixed odd $n\ge3$, the leading terms of the genus and the formulas above give $a(\mathcal G_{q,n})/g(\mathcal G_{q,n})\to1/4$ as $q=2^h\to\infty$.

The other limit behaves differently. For $n\ge5$, equation \eqref{eq:dnshift} places every $d_n$ in the single residue class $d_0+N$ modulo $2N$. Corollary~\ref{cor:periodicity}, with $Y=\mathcal H_q$, $N=q^3$, and $S=\Sigma_q$, therefore gives, for fixed $q$,
\begin{equation}\label{eq:GGSlimitn}
\lim_{n\to\infty}\frac{a(\mathcal G_{q,n})}{g(\mathcal G_{q,n})}
=\frac{\Sigma_q}{NA}
=\frac{3q^5+q^3+2q^2-q-2}{12q^2(q+1)^2(q-1)}.
\end{equation}
Since $3q^5+q^3+2q^2-q-2-3q^2(q+1)^2(q-1)=-(q-2)(3q^3+2q^2-q-1)$, this limit equals $1/4$ for $q=2$ and is strictly smaller than $1/4$ for $q\ge4$; for instance it is $527/2400\approx0.2196$ for $q=4$.

A comparison with the same-genus Beelen--Montanucci family is given in Section~8.

\section{The Beelen--Montanucci curves}\label{sec:BM}
We apply Theorem~\ref{thm:multitransfer} to \eqref{eq:BMmodel}. We use the geometry described in \cite{BMnew,MPL} and first calculate the divisor of $dz$.

\subsection{Geometry and the half-canonical divisor}
Let $Y$ be the Hermitian curve $y^{q+1}=x^{q+1}+1$, and let $I=Q_1+\cdots+Q_{q+1}$ be its divisor at infinity. Set
\[
{N_B}=q^2-q,\qquad f=y\frac{x^{q^2}+x}{x^{q+1}+1}
=\frac{x^{q^2}+x}{y^q}.
\]
The polynomial $x^{q+1}+1$ divides $x^{q^2}+x$, so $f$ is regular on the affine curve. Both $x$ and $y$ have pole order one at each $Q_i$, giving $(f)_\infty=N_B I$. Differentiation in characteristic two and the plane adjunction formula give
\begin{equation}\label{eq:BMdf}
df=\frac{dx}{y^q},\qquad \divis_Y(df)=(q-2)I.
\end{equation}
Indeed, $dx/y^q=dx/F_y$ for the smooth plane equation $F=y^{q+1}+x^{q+1}+1$ of degree $q+1$; it is a nowhere vanishing differential on the affine curve and has order $q-2$ at each point of the transverse line at infinity.

For any positive odd $m$ coprime to ${N_B}$, let $B_{q,m}$ be the smooth projective curve defined by $z^m=f$, and let ${E_B}=P_1+\cdots+P_{q+1}$ be the divisor above $I$. Theorem~\ref{thm:multitransfer} yields
\begin{equation}\label{eq:BMbm}
\divis(dz)=2b_m E_B,\qquad
b_m=\frac{m(q^2-1)-{N_B}-1}{2},
\qquad a(B_{q,m})=\ell_{B_{q,m}}(b_m E_B).
\end{equation}
The genus is $(m(q+1)(q^2-1)-q^3+1)/2$. At $m=d_n$, the coefficient $2b_m$ becomes $q^{n+1}-q^n-q^2+2q-2$, agreeing with \cite[Lemma~2.3 and Section~4.1]{MPL}.

\subsection{The floor sum}
Write $F_q(k)=\ell_Y(kI)$, with $F_q(k)=0$ for $k<0$. The plane embedding gives, for $k\ge0$,
\begin{equation}\label{eq:BMhilbert}
F_q(k)=\binom{k+2}{2}-
\begin{cases}
\binom{k-q+1}{2},&k\ge q+1,\\
0,&0\le k\le q.
\end{cases}
\end{equation}
To justify the formula, use the exact sequence obtained by restricting $\mathcal O_{\mathbb P^2}(k)$ to the plane curve of degree $q+1$, together with $H^1(\mathbb P^2,\mathcal O(t))=0$. In particular,
\[
F_q(k)-F_q(k-1)=\min\{k+1,q+1\}\quad(k\ge0).
\]

\begin{proposition}\label{prop:BMfloor}
For positive odd $m$ coprime to ${N_B}=q^2-q$,
\begin{equation}\label{eq:BMfloor}
a(B_{q,m})=
\sum_{i=0}^{q^2/2-1}\min\{i+1,q+1\}
\left\lfloor\frac{m(q^2-1-2i)+{N_B}-1}{2{N_B}}\right\rfloor.
\end{equation}
\end{proposition}
\begin{proof}
The multipoint transfer formula first gives
\[
a(B_{q,m})=\sum_{j=0}^{m-1}
F_q\!\left(\left\lfloor\frac{b_m-jN_B}{m}\right\rfloor\right).
\]
Expanding $F_q$ into its dimension jumps and reversing the finite sums, the coefficient of $\min\{i+1,q+1\}$ is the number of $j\in\{0,\ldots,m-1\}$ with $jN_B\le b_m-mi$. Since
\[
\frac{b_m}{m}<\frac{q^2-1}{2},\qquad
2(mN_B-b_m)=m(q-1)^2+{N_B}+1>0,
\]
only $0\le i\le q^2/2-1$ can contribute, and the upper bound $j<m$ is automatic whenever the count is nonzero. The count is therefore
\[
\max\left\{0,1+\left\lfloor\frac{b_m-mi}{{N_B}}\right\rfloor\right\}
=\left\lfloor\frac{m(q^2-1-2i)+{N_B}-1}{2{N_B}}\right\rfloor.
\]
The last floor is nonnegative because $q^2-1-2i\ge1$ in this range. This proves the formula.
\end{proof}

\subsection{Evaluation at the Beelen--Montanucci degrees}
\begin{proof}[Proof of Theorem~\ref{thm:BM}]
Put $m_3=q^2-q+1={N_B}+1$. For odd $n\ge3$,
\[
m=d_n=1+{N_B}(1+q^2+\cdots+q^{n-3}),
\]
so $m=m_3+2N_B u$ for an integer $u\ge0$. In particular $(m,{N_B})=1$. In \eqref{eq:BMfloor}, increasing $m$ by $2{N_B}$ adds $q^2-1-2i$ to the $i$th floor. Hence
\begin{equation}\label{eq:BMshift}
a(\mathcal B_{q,n})=a(\mathcal B_{q,3})+
\frac{m-m_3}{2{N_B}}\,S_q^{\rm BM},
\end{equation}
where
\begin{align*}
S_q^{\rm BM}
&=\sum_{i=0}^{q^2/2-1}\min\{i+1,q+1\}(q^2-1-2i)\\
&=\frac{(q+1)q^4}{4}
-\sum_{i=0}^{q-1}(q-i)(q^2-1-2i)\\
&=\frac{(q+1)q^4}{4}
-\frac{q(q+1)(3q^2-2q-1)}6\\
&=\frac{q(q+1)(3q^3-6q^2+4q+2)}{12}.
\end{align*}
The middle equality follows by substituting $j=q-i$ and summing $j$ and $j^2$. Since $\mathcal B_{q,3}$ is the GK curve, Proposition~\ref{prop:GK} supplies the initial value in \eqref{eq:BMshift}. Substituting $m=(q^n+1)/(q+1)$ gives
\[
a(\mathcal B_{q,n})=
\frac{q(q+1)(3q^3-6q^2+10q-4)}{24}
+\frac{(q^n-q^3)(3q^3-6q^2+4q+2)}{24(q-1)},
\]
which simplifies to \eqref{eq:BManswer}. The leading terms establish the limit for fixed $n$.
\end{proof}

For fixed $q$, there is a different limiting ratio:
\[
\lim_{\substack{n\to\infty\\n\text{ odd}}}
\frac{a(\mathcal B_{q,n})}{g(\mathcal B_{q,n})}
=\frac{3q^3-6q^2+4q+2}{12(q-1)^2(q+1)}.
\]
For $q=2$ it is $5/18$, whereas the corresponding GGS limit is $1/4$. The difference in the $a$-numbers also distinguishes the underlying abelian varieties, as follows.

\begin{corollary}\label{cor:JacobianComparison}
For $q=2^h$ and odd $n\ge3$, the Jacobians of $\mathcal G_{q,n}$ and $\mathcal B_{q,n}$ are isogenous over $\F_{q^{2n}}$. Whenever their $a$-numbers differ, these Jacobians are not isomorphic even as unpolarized abelian varieties over $k$. In particular, this occurs for $q=2$ and $n=5$, where
\[
g=46,\qquad a(\mathcal G_{2,5})=12,\qquad a(\mathcal B_{2,5})=14.
\]
\end{corollary}
\begin{proof}
The two curves have the same genus and are maximal over $\F_{q^{2n}}$. Hence their Jacobians have the same Frobenius characteristic polynomial $(T+q^n)^{2g}$. Tate's isogeny theorem over finite fields \cite{Tate} gives the asserted isogeny. An isomorphism of the underlying abelian varieties would induce an isomorphism of their $2$-torsion group schemes and preserve the $a$-number. The displayed values follow from Theorems~\ref{thm:GGS} and~\ref{thm:BM}.
\end{proof}

The nonisomorphism of the curves for $n\ge5$ is already known from \cite[Corollary~2.6]{BMnew}; the corollary concerns their Jacobians without a chosen polarization.

\begin{center}
\begin{tabular}{rrrrr}
\toprule
$q$ & $n$ & genus & $a(\mathcal B_{q,n})$ & $a(\mathcal G_{q,n})$\\
\midrule
2 & 3 & 10 & 4 & 4\\
2 & 5 & 46 & 14 & 12\\
2 & 7 & 190 & 54 & 48\\
4 & 3 & 456 & 110 & 110\\
4 & 5 & 7\,656 & 1\,630 & 1\,686\\
4 & 7 & 122\,856 & 25\,950 & 26\,982\\
8 & 5 & 1\,031\,968 & 231\,396 & 234\,508\\
\bottomrule
\end{tabular}
\end{center}

\begin{remark}
For $n\ge5$, the distinguished rational points of $\mathcal B_{q,n}$ have nonsymmetric semigroups \cite[Corollary~5.3]{MPL} and hence are not subcanonical. The divisor in \eqref{eq:BMbm} is supported at the $q+1$ points above $I$.
\end{remark}

\section{Odd-degree radical covers of the Suzuki curve}
Let $r=2^s$, $q=2r^2$, and let $S_q$ and ${\Delta_S}=q+2r+1$ be as in Section~\ref{sec:Suzuki}. For every odd $m\ge1$, consider
\[
X_{q,m}:\quad y^q+y=x^r(x^q+x),\qquad t^m=x^q+x.
\]
Apply Theorem~\ref{thm:transfer} with $Y=S_q$, $f=x^q+x$, and $N=q^2$. Since $df=dx$ and the pole order of $x^q+x$ is $q^2$, its hypotheses hold. We obtain
\[
g_m=m r(q-1)+\frac{(m-1)(q^2-1)}2,
\]
\[
H_{X_{q,m}}(P_\infty)=
\langle mq,m(q+r),m(q+2r),m\Delta_S,q^2\rangle,
\]
\begin{equation}\label{eq:generalSuzukiCover}
a(X_{q,m})=\sum_{j=0}^{m-1}
\ell_{S_q}\!\left(\left\lfloor\frac{g_m-1-jq^2}{m}\right\rfloor Q_\infty\right).
\end{equation}
The formula holds for every odd $m$ and is affine on each odd residue class modulo $2q^2$ by Corollary~\ref{cor:periodicity}. We now evaluate it at the Skabelund degree $m=q-2r+1$.

\subsection{The Skabelund specialization}

Return to $r=q_0=2^s$ and $q=2r^2$. We retain the Suzuki notation ${\Delta_S}$, $g_S$, $H_S$, and $\mathcal L_S(n)$ from Section~\ref{sec:Suzuki}, and put $m=q-2r+1$, which is odd. Then $m\Delta_S=q^2+1$ and $X_q=X_{q,m}$. Let $P_\infty$ denote the unique point at infinity of $X_q$, and put $g=g(X_q)$; the formula for $g_m$ above gives \eqref{eq:genus}. The preceding application of Theorem~\ref{thm:transfer} gives
\begin{align}
\divis_{X_q}(dt)&=(2g-2)P_\infty,\label{eq:dt}\\
H_X:=H_{X_q}(P_\infty)&=\langle mq,m(q+r),m(q+2r),q^2,m\Delta_S\rangle.\label{eq:skabsg}
\end{align}
Skabelund \cite{Skabelund} stated the generators in \eqref{eq:skabsg}; Beelen, Landi, and Montanucci proved the semigroup equality \cite[Theorem~3.2]{BLM} and computed the exact differential divisor in \cite[proof of Corollary~2.2]{BLM}. Here both identities also follow from Theorem~\ref{thm:transfer}. Proposition~\ref{prop:method} gives
\begin{equation}\label{eq:count}
a(X_q)=\ell_{X_q}((g-1)P_\infty)
=\#\bigl(H_X\cap[0,g-1]\bigr).
\end{equation}
By the proof of Theorem~\ref{thm:transfer}, every element of $H_X$ has a unique expression $mv+jq^2$ with $v\in H_S$ and $0\le j<m$.

\begin{proposition}\label{prop:sum}
One has
\begin{equation}\label{eq:suzukisum}
a(X_q)=\sum_{n=0}^{r^2-1}\mathcal L_S(n).
\end{equation}
\end{proposition}
\begin{proof}
The identities $q=2r^2$ and \eqref{eq:genus} give
\[
g-1=(q^2+1)(r^2-1)=m\Delta_S(r^2-1).
\]
Since $q^2=m\Delta_S-1$ and $0\le j<m$, the floor in \eqref{eq:generalSuzukiCover} is exactly
\[
\left\lfloor\frac{g-1-jq^2}{m}\right\rfloor={\Delta_S}(r^2-1-j).
\]
Therefore
\[
a(X_q)=\sum_{j=0}^{m-1}\mathcal L_S(r^2-1-j).
\]
We have $m-1=2r^2-2r\ge r^2-1$, because $r\ge2$. The terms with negative argument vanish, and the remaining arguments run through $0,\ldots,r^2-1$.
\end{proof}

\subsection{Evaluation at the Skabelund degree}
The small-degree terms were computed in Lemma~\ref{lem:small}. We evaluate the remaining terms by Riemann--Roch.

\begin{proof}[Proof of Theorem~\ref{thm:main}]
Since
\[
g_S-1=(r-1){\Delta_S},
\]
equation \eqref{eq:suzukican} and Riemann--Roch imply, for every integer $n$,
\begin{equation}\label{eq:RR}
\mathcal L_S(n)={\Delta_S}(n-r+1)+\mathcal L_S(2r-2-n).
\end{equation}
Split the sum in Proposition~\ref{prop:sum} at $n=r-1$ and apply \eqref{eq:RR} to the terms $r\le n\le r^2-1$. The reflected terms have arguments $r-2,r-3,\ldots,-(r-1)^2$. Hence
\begin{equation}\label{eq:finalsum}
a(X_q)=2\sum_{n=0}^{r-2}\mathcal L_S(n)+\mathcal L_S(r-1)
+{\Delta_S}\sum_{u=1}^{r^2-r}u.
\end{equation}
By Lemma~\ref{lem:small} and the standard binomial summation identities,
\begin{align*}
\sum_{n=0}^{r-2}\mathcal L_S(n)
&=\binom{r+2}{4}+\binom{r+1}{4}
=\frac{r^2(r^2-1)}{12},\\
\mathcal L_S(r-1)&=\frac{r(r+1)(2r+1)}6.
\end{align*}
Thus \eqref{eq:finalsum} yields
\begin{equation}\label{eq:factored}
a(X_q)=\frac{{\Delta_S}(r^2-r)(r^2-r+1)}2
+\frac{r(r+1)(r^2+r+1)}6.
\end{equation}
Substituting ${\Delta_S}=2r^2+2r+1$ and expanding proves \eqref{eq:answer}. Since
\[
g=g(X_q)=4r^6-4r^4+r^2,
\]
subtracting \eqref{eq:answer} gives the asserted Cartier rank. For the bounds, write $a=a(X_q)$. Direct expansion gives
\[
3\Bigl(a-\frac g4\Bigr)=-r^4(3r-5)+r^3+\frac{r^2}4-r,
\qquad
3\Bigl(a-\frac g4+\frac g{4r}\Bigr)=r^3(5r-2)+\frac{r(r-1)}4.
\]
For $r\ge2$ one has $r^4(3r-5)\ge r^4>r^3+r^2/4$, so the first quantity is negative, while the second is positive. Dividing by $3g$ proves \eqref{eq:Skabbounds}, and the limit follows.
\end{proof}

\subsection{Exact differentials and examples}\label{sec:Skabexamples}
The argument also gives a basis of the Cartier kernel. On $S_q$ set
\[
z=y^{2r}+x^{2r+1},\qquad w=xy^{2r}+z^{2r}.
\]
These are the functions $h_1,h_2$ of \cite[Section~3.2]{FGMPW} in our coordinates. Their pole divisors are $(q+2r)Q_\infty$ and $\Delta_S Q_\infty$, respectively; those of $x,y$ are $qQ_\infty,(q+r)Q_\infty$ \cite[Lemma~3.4]{FGMPW}. Since $P_\infty$ is totally ramified over $Q_\infty$, their pullbacks to $X_q$ have these pole orders multiplied by $m$, whereas $t$ has pole order $q^2$.

\begin{corollary}\label{cor:basis}
Let $\mathcal I$ consist of the tuples $(\alpha,\beta,\gamma,\delta,j)$ satisfying
\[
\alpha\ge0,\quad \beta\in\{0,1\},\quad 0\le \gamma,\delta<r,\quad 0\le j<m,
\]
and
\[
m\bigl(\alpha q+\beta(q+r)+\gamma(q+2r)+\delta {\Delta_S}\bigr)+jq^2\le g(X_q)-1.
\]
Then a $k$-basis of $\ker\Cart$ on $\HH(X_q,\Omega_{X_q}^1)$ is
\[
\left\{x^{2\alpha}y^{2\beta}z^{2\gamma}w^{2\delta}t^{2j}\,dt:
(\alpha,\beta,\gamma,\delta,j)\in\mathcal I\right\}.
\]
\end{corollary}
\begin{proof}
The unique representation $mv+jq^2$ of the elements of $H_X$ and Lemma~\ref{lem:normal} show that the functions $x^\alpha y^\beta z^\gamma w^\delta t^j$ indexed by $\mathcal I$ have distinct pole orders, one for each nongap at most $g-1$. They therefore form a basis of $L_{X_q}((g-1)P_\infty)$. Lemma~\ref{lem:cartier} gives the claimed basis after squaring and multiplying by $dt$.
\end{proof}

For the first values of $s$, Theorem~\ref{thm:main} gives the following table.
\begin{center}
\begin{tabular}{rrrr}
\toprule
$q_0$ & $q$ & $g(X_q)$ & $a(X_q)$\\
\midrule
2 & 8 & 196 & 46\\
4 & 32 & 15\,376 & 3\,268\\
8 & 128 & 1\,032\,256 & 232\,296\\
16 & 512 & 66\,846\,976 & 15\,773\,776\\
32 & 2048 & 4\,290\,774\,016 & 1\,040\,897\,696\\
\bottomrule
\end{tabular}
\end{center}

\begin{remark}
The middle term in \eqref{eq:finalsum} is the known Suzuki $a$-number:
\[
\mathcal L_S(r-1)=a(S_q)=\frac{r(r+1)(2r+1)}6;
\]
see \cite[Theorem 4.1]{FGMPW}. The present calculation requires the sum of Riemann--Roch dimensions in \eqref{eq:suzukisum}, rather than only this single value. The bounds \eqref{eq:Skabbounds} parallel those of \cite[Theorem~4.1]{FGMPW}, namely $1/6<a(S_q)/g(S_q)<1/6+2^{-(s+1)}$: the relative $a$-number of the Suzuki curve approaches $1/6$ from above, whereas that of its Skabelund cover approaches $1/4$ from below.
\end{remark}

\begin{remark}[Ekedahl--Oort information]\label{rem:EO}
The curves $\mathcal G_{q,n}$, $\mathcal B_{q,n}$, and $X_q$ are maximal over finite fields, hence supersingular, and so have $2$-rank $0$. As in \cite[Section~5]{FGMPW}, their $a$-numbers give partial information on the Ekedahl--Oort type $[\nu_1,\ldots,\nu_g]$ of the $2$-torsion group scheme of the Jacobian: $\nu_1=0$ and $\nu_g=g-a=\operatorname{rank}\Cart$. Since $a<g$ in every case above, none of these Jacobians is superspecial.

There is also a bound on product decompositions of the underlying abelian variety, as in \cite[Lemma~5.2]{FGMPW}. If
\[
\operatorname{Jac}(X)\simeq A_1\times\cdots\times A_t,
\qquad \dim A_i>0,
\]
is an isomorphism of abelian varieties, then each $A_i$ has $2$-rank zero and hence $a(A_i)\ge1$. Additivity of the $a$-number gives $t\le a(X)$. The isomorphism need not preserve a polarization. For comparison, a smooth connected curve's Jacobian is indecomposable as a principally polarized abelian variety.
\end{remark}

\section{Fibre products of Artin--Schreier curves}\label{sec:fibre}
Van der Geer and van der Vlugt \cite{GV95} used fibre products of Artin--Schreier curves to construct supersingular and maximal curves. We consider polynomial equations, so that infinity is the only branch point. After the general reduction, we compute the $a$-numbers of common-monomial fibre products and of a biquadratic family with two distinct pole orders. A further genus-six family has varying $a$-number and constant quadratic quotient genera and $a$-numbers.

\subsection{Polynomial Artin--Schreier fibre products}
Call a polynomial reduced if it is a sum of monomials of odd positive degree, with no constant term. Let ${\mathcal D}\subset k[x]$ be an $e$-dimensional $\F_2$-space of reduced polynomials. Assume $e\ge1$, and choose a basis $f_1,\ldots,f_e$. Let $C_{\mathcal D}$ be the smooth projective curve with function field
\begin{equation}\label{eq:fibreModel}
k(C_{\mathcal D})=k(x,y_1,\ldots,y_e),\qquad y_i^2+y_i=f_i(x).
\end{equation}
Every nonzero element of ${\mathcal D}$ has odd positive degree, so its Artin--Schreier class is nonzero. Thus $C_{\mathcal D}$ is connected, the cover of the $x$-line has degree $2^e$, and changing the chosen basis does not change the curve over the $x$-line \cite[Lemma~2.1]{GV95}. Write $d(f)=\deg f$ for $0\ne f\in {\mathcal D}$.

\begin{proposition}\label{prop:fibreGeneral}
The curve $C_{\mathcal D}$ has a unique point $P$ above infinity and is unramified elsewhere. Moreover,
\begin{equation}\label{eq:fibreGenus}
g(C_{\mathcal D})=\frac12\sum_{0\ne f\in {\mathcal D}}(d(f)-1),\qquad
\divis(dx)=(2g(C_{\mathcal D})-2)P.
\end{equation}
Consequently,
\begin{equation}\label{eq:fibreA}
a(C_{\mathcal D})=\#\bigl(H_{C_{\mathcal D}}(P)\cap[0,g(C_{\mathcal D})-1]\bigr).
\end{equation}
\end{proposition}
\begin{proof}
The affine cover in \eqref{eq:fibreModel} is finite \'etale over the affine line, since the derivative of each equation in its own $y_i$ is $1$. Every quadratic subextension has equation $y^2+y=f(x)$ for some $0\ne f\in {\mathcal D}$, hence is ramified at infinity. If the inertia group at infinity were a proper subgroup of $(\mathbb Z/2\mathbb Z)^e$, a nonzero character vanishing on it would give a quadratic subextension unramified at infinity. This contradiction proves total ramification and uniqueness of $P$.

The genus identity is \cite[Corollary~2.6]{GV95}, applied to the quadratic quotients, each of genus $(d(f)-1)/2$. Alternatively it follows from the conductor--different formula, whose nontrivial characters have conductors $d(f)+1$. Proposition~\ref{prop:wildpreserve}, applied to the cover of $\mathbb P^1$, gives the differential divisor, and Lemma~\ref{lem:cartier} proves the count. The same convention gives zero when the genus is zero.
\end{proof}

Over finite fields, \cite[Proposition~2.5]{GV95} gives an isogeny
\[
\operatorname{Jac}(C_{\mathcal D})\sim
\prod_{0\ne f\in {\mathcal D}}\operatorname{Jac}(C_f),
\qquad C_f:\ y^2+y=f(x).
\]
The $a$-number is not an isogeny invariant. In Proposition~\ref{ex:fibreVariation}, all three quadratic quotients have genus two and $a$-number one, while the $a$-number of the fibre product varies.

\subsection{Common-monomial fibre products}
\begin{proposition}\label{prop:commonMonomial}
Let $A\subset k$ be an $e$-dimensional $\F_2$-space, let $M\ge1$ be odd, and let ${\mathcal D}=A x^M$. Put ${Q_A}=2^e$. Then
\[
g(C_{\mathcal D})=\frac{({Q_A}-1)(M-1)}2,\qquad
H_{C_{\mathcal D}}(P)=\langle {Q_A},M\rangle,
\]
and
\begin{equation}\label{eq:commonMonomialA}
a(C_{\mathcal D})=\sum_{j=0}^{{Q_A}/2-1}
\left\lfloor\frac{(2j+1)M+{Q_A}-1}{2{Q_A}}\right\rfloor.
\end{equation}
In particular, the answer is independent of the coefficient space $A$ once $e$ and $M$ are fixed.
\end{proposition}
\begin{proof}
The genus follows from \eqref{eq:fibreGenus}, since every nonzero element of ${\mathcal D}$ has degree $M$. To find a function of pole order $M$, first consider the auxiliary cover of the $s$-line with equations $u_i^2+u_i=a_i s$, where the $a_i$ form a basis of $A$. Proposition~\ref{prop:fibreGeneral} gives genus zero and a unique point at infinity. Choose a rational coordinate $t$ with a simple pole there. Then $s={\mathcal L}(t)$ is a polynomial of degree ${Q_A}$, and every $u_i$ is a polynomial in $t$.

The deck group fixes infinity, so it acts on $t$ by affine transformations. Each element has order dividing two; in characteristic two its linear coefficient must be one. Hence this group is a translation group $V\subset k$ of order ${Q_A}$. A fibre of $s$ shows that ${\mathcal L}(t)$ is, after translating $t$ and adjusting its constant term, a scalar multiple of $\prod_{v\in V}(t-v)$. This polynomial is additive and has nonzero linear coefficient. Pulling back by $s=x^M$ therefore gives a model
\[
{\mathcal L}(t)=x^M,\qquad \deg {\mathcal L}={Q_A}.
\]
This pullback is connected: in a field generated by one root $t$, let $e_P$ be the ramification index of a point above infinity and let $b$ be the pole order of $t$ there. Comparing pole orders in ${\mathcal L}(t)=x^M$ gives $Q_A b=e_P M$, so $Q_A\mid e_P$ because $(Q_A,M)=1$. Since $e_P\le[k(x,t):k(x)]\le Q_A$, the extension has degree $Q_A$ and the defining polynomial is irreducible. It has the same function field as \eqref{eq:fibreModel}. The pole orders of $x,t$ are ${Q_A},M$, so $\langle {Q_A},M\rangle\subset H(P)$. Their genera agree, proving equality. Formula~\eqref{eq:commonMonomialA} is now the same Ap\'ery count used in Proposition~\ref{prop:additive}, with ${Q_A}$ in place of $q$.
\end{proof}

For $M>1$, the model $x^M={\mathcal L}(t)$ also realizes $C_{\mathcal D}$ as a tame cyclic cover of the $t$-line of degree $M$. Since ${\mathcal L}$ is separable of degree ${Q_A}$, its branch exponents are one at each of its ${Q_A}$ roots and $-{Q_A}\bmod M$ at infinity. Thus Elkin's formula \eqref{eq:Elkinrank} also applies to this subfamily and gives a second explanation of coefficient independence. The semigroup proof above directly evaluates the invariant as \eqref{eq:commonMonomialA}; the case $M=1$ is rational.

The common-monomial maximal curves constructed in \cite[Theorem~3.3]{GV95} have $M=2^h+1$, so the preceding proposition computes their $a$-numbers for every rank $e$ allowed by that construction. In particular,
\begin{equation}\label{eq:fibreSpecial}
a(C_{\mathcal D})=
\begin{cases}
2^{h+e-3},&h>e,\\[2pt]
2^e(2^e+2)/8,&h=e.
\end{cases}
\end{equation}
For $h<e$, formula~\eqref{eq:commonMonomialA} remains valid without any further assumption. These specializations follow from Corollary~\ref{cor:additiveprogressions}: for $h>e$ its parameter $\lambda=2^{h-e}$ is even, and for $h=e$ it equals one. Maximality here is supplied by the construction in \cite{GV95}.

\begin{center}
\begin{tabular}{rrrr}
\toprule
$e$ & $h$ & genus for $M=2^h+1$ & $a(C_{\mathcal D})$\\
\midrule
1 & 2 & 2 & 1\\
2 & 1 & 3 & 1\\
2 & 2 & 6 & 3\\
2 & 3 & 12 & 4\\
3 & 2 & 14 & 5\\
3 & 3 & 28 & 10\\
3 & 4 & 56 & 16\\
\bottomrule
\end{tabular}
\end{center}

\subsection{Two distinct monomial pole orders}
\begin{proposition}\label{prop:biquadratic}
Let $d<e$ be odd positive integers, and let $C_{d,e}$ be the smooth projective model of
\[
y^2+y=x^d,\qquad z^2+z=x^e.
\]
Then
\[
g=\frac{d+2e-3}{2},\qquad
H(P)=\langle4,2d,2e-d\rangle.
\]
With $B=g-1$, its $a$-number is
\begin{equation}\label{eq:biquadraticA}
a(C_{d,e})=
\sum_{w\in\{0,\,2d,\,2e-d,\,2e+d\}}
\max\left\{0,1+\left\lfloor\frac{B-w}{4}\right\rfloor\right\}.
\end{equation}
\end{proposition}
\begin{proof}
The three quadratic quotients have pole orders $d,e,e$, so Proposition~\ref{prop:fibreGeneral} gives the genus and a unique point $P$ at infinity. Put $t=(e-d)/2$ and $w=z+x^t y$. Then
\[
w^2+w=(x^{2t}+x^t)y.
\]
Since $x$ and $y$ have pole orders $4,2d$, respectively, the right side has pole order $8t+2d$. Thus $w$ has pole order $4t+d=2e-d$. We obtain $H'=\langle4,2d,2e-d\rangle\subset H(P)$.

The Ap\'ery representatives of $H'$ modulo $4$ are
\[
0,\qquad 2d,\qquad 2e-d,\qquad 2e+d.
\]
To check this, put $c=2e-d$, which is odd and greater than $d$. The relation $2c=2d+4(e-d)$ reduces its coefficient modulo two, and $2(2d)=4d$ reduces the coefficient of $2d$ modulo two. The four resulting representatives have distinct residues modulo four. Their genus is
\[
\frac{2d+(2e-d)+(2e+d)}4-\frac32
=\frac{d+2e-3}{2}=g.
\]
Therefore $H'=H(P)$, and the exact-differential count gives \eqref{eq:biquadraticA}.
\end{proof}

When $d=2^u+1$ and $e=2^v+1$ with $1\le u<v$, these are supersingular fibre products of the type in \cite[Proposition~3.1]{GV95}. For instance, $(d,e)=(3,5)$ gives genus $5$, semigroup $\langle4,6,7\rangle$, and $a=2$.

\subsection{Coefficient dependence at fixed quotient data}
\begin{proposition}\label{ex:fibreVariation}
Fix $\alpha\in\F_4\setminus\F_2$, so $\alpha^2+\alpha+1=0$. For $\lambda\in k$, let $C_\lambda$ be the smooth projective fibre product
\begin{equation}\label{eq:fibreVariation}
y^2+y=x^5,\qquad
z^2+z=\alpha x^5+\lambda x^3.
\end{equation}
Every $C_\lambda$ has genus $6$ and is supersingular. Its $a$-number is
\begin{equation}\label{eq:fibreJump}
a(C_\lambda)=
\begin{cases}
3,&\lambda=0,\\
2,&\lambda\ne0.
\end{cases}
\end{equation}
For all $\lambda$, the three quadratic quotients have genus two and $a$-number one.
\end{proposition}
\begin{proof}
All three nonzero combinations of the right sides in \eqref{eq:fibreVariation} have degree five. Proposition~\ref{prop:fibreGeneral} gives $g=6$, total ramification at the unique point $P$ above infinity, and $\divis(dx)=10P$. Every quadratic quotient has $a=\lfloor(5+1)/4\rfloor=1$ by Proposition~\ref{prop:additive} with $q=2$. They are supersingular curves of the form $y^2+y=xR(x)$; \cite[Proposition~3.1]{GV95} gives supersingularity of their fibre product. Since every $\lambda\in k$ belongs to a finite field, the finite-field statement applies to each member.

Set $w=z+\alpha^2y$. Then
\[
w^2+w=y+\lambda x^3.
\]
If $\lambda=0$, elimination of $y$ gives $w^4+w=x^5$. Hence $H(P)=\langle4,5\rangle$ and its nongaps up to $g-1=5$ are $0,4,5$, giving $a=3$.

Suppose $\lambda\ne0$. The pole orders of $x,y$ are $4,10$. The equation for $w$ therefore gives pole order $6$ for $w$. Choose $c\in k$ with $c^2\lambda=1$ and put $u=y+cxw$. Then
\[
u^2=y+c^2x^2w+c^2x^2y.
\]
The last summand has pole order $18$, strictly larger than those of the other two summands, so $u$ has pole order $9$. Thus $\langle4,6,9\rangle\subset H(P)$. Its Ap\'ery set modulo four is $\{0,6,9,15\}$, of genus $6$, so equality follows. Its only nongaps up to five are $0,4$, proving $a=2$.
\end{proof}

Both strata in \eqref{eq:fibreJump} have the same three quadratic quotient genera and $a$-numbers, and the same upper ramification break $5$ for every nontrivial character. Nevertheless, the $a$-number of the fibre product changes. In particular, one cannot recover it by summing the quotient $a$-numbers or by using only their pole orders.

For more general spaces of polynomials $xR(x)+ax$ in \cite[Section~3]{GV95}, one must first account for a possible loss of rank under Artin--Schreier reduction. Let ${\mathcal D}_0\subset k[x]$ be an $e$-dimensional $\F_2$-space, put $\wp(h)=h^2+h$, and let $\overline {\mathcal D}_0$ be its image in $k(x)/\wp(k(x))$, of dimension $e'$. Choose reduced polynomial representatives for a basis of $\overline {\mathcal D}_0$. If $e'>0$, Proposition~\ref{prop:fibreGeneral} applies to the connected cover defined by this basis, with covering degree $2^{e'}$. After a change of basis and translations of the $y_i$, the other $e-e'$ equations become $u_i^2+u_i=0$. Over $k$, the original fibre product therefore has $2^{e-e'}$ components, whose smooth projective models are isomorphic to that connected cover. If $e'=0$, all components are rational. The genus and semigroup formulas must use the reduced rank and a connected component.

For example,
\[
y_1^2+y_1=x^3,\qquad y_2^2+y_2=x^3+x^2+x
\]
has $e=2$ and $e'=1$: setting $u=y_1+y_2+x$ gives $u^2+u=0$. It consists of two copies of the genus-one curve $y_1^2+y_1=x^3$. The pole semigroup must then be computed on a connected component.

\subsection{Ekedahl--Oort information}
\begin{remark}[Supersingularity and the final type]\label{rem:fibreEO}
Let $C$ be a supersingular member of the families above, of genus $g\ge1$, and put $a=a(C)$. Supersingularity fixes its Newton polygon, whose slopes are all $1/2$, but does not determine the Ekedahl--Oort type of $\operatorname{Jac}(C)[2]$. Write this type as the final sequence $\nu=(\nu_1,\ldots,\nu_g)$, with $\nu_0=0$. The standard relations are
\[
\nu_i\le\nu_{i+1}\le\nu_i+1,\qquad
f=\max\{i:\nu_i=i\},\qquad a=g-\nu_g,
\]
where $f$ is the $2$-rank; see \cite[Section~2.1.5]{PW}. Since $f=0$, our $a$-number formulas imply
\begin{equation}\label{eq:EObounds}
\nu_1=0,\qquad \nu_g=g-a,\qquad
\max\{0,i-a\}\le\nu_i\le\min\{i-1,g-a\}
\quad(1\le i\le g).
\end{equation}
These bounds apply to the supersingular subfamilies above. The general curves in Proposition~\ref{prop:fibreGeneral} need not be supersingular.

More generally, let $\eta_j=\operatorname{rank}(\Cart^j)$ and $\eta_0=g$. The final filtration contains the successive Cartier images of $H^0(C,\Omega_C^1)$, and hence
\begin{equation}\label{eq:EOiterates}
\nu_{\eta_j}=\eta_{j+1}\qquad\text{whenever }\eta_j>0.
\end{equation}
Indeed, $H^0(C,\Omega_C^1)=V H^1_{\mathrm{dR}}(C)$, and $V$ restricts to Cartier on this subspace. The iterates of $V$ belong to the canonical filtration underlying the final type. Even the entire sequence $(\eta_j)$ need not determine $\nu$; see \cite[Section~2.2.3]{PW}. For the genus-six family, we compute the additional de Rham data in the next proposition.
\end{remark}

\begin{proposition}[Ekedahl--Oort types of the genus-six family]\label{prop:fibreEO}
For the family $C_\lambda$ in Proposition~\ref{ex:fibreVariation}, the member at $\lambda=0$ has
\begin{equation}\label{eq:EOzero}
\operatorname{EO}(\operatorname{Jac}(C_0))=(0,0,0,1,2,3).
\end{equation}
For every $\lambda\ne0$, the ranks of the positive Cartier iterates are
\begin{equation}\label{eq:EOCartierRanks}
(\eta_1,\eta_2,\eta_3,\eta_4)=(4,2,1,0),
\end{equation}
and
\begin{equation}\label{eq:EOgeneric}
\operatorname{EO}(\operatorname{Jac}(C_\lambda))=(0,1,1,2,3,4).
\end{equation}
In particular, the family has constant supersingular Newton polygon, while its Ekedahl--Oort type is constant on $\lambda\ne0$ and changes at $\lambda=0$.
\end{proposition}
\begin{proof}
At $\lambda=0$, the curve has equation $w^4+w=x^5$. The six forms
\[
dx,\quad x\,dx,\quad x^2\,dx,\quad w\,dx,\quad xw\,dx,\quad w^2\,dx
\]
form a basis: their coefficient functions have distinct pole orders $0,4,8,5,9,10$, and $\divis(dx)=10P$. Their Cartier images, in the same order, are
\[
0,\quad dx,\quad 0,\quad x^2\,dx,\quad w^2\,dx,\quad0.
\]
Thus $\eta_1=3$ and $\eta_2=0$. Equation~\eqref{eq:EOiterates} gives $\nu_6=3$ and $\nu_3=0$, which force \eqref{eq:EOzero}. This is the $q=4$ Hermitian case of the known computation of Pries and Weir \cite[Section~2.2.2]{PW}.

For $\lambda\ne0$, retain $w,u,c$ from the proof of Proposition~\ref{ex:fibreVariation}, so that $c^2\lambda=1$ and $u=y+cxw$. Put $\rho^2=\lambda$, $\sigma^2=c$, and $\tau=w+y$. A basis of regular differentials is now
\[
dx,\quad x\,dx,\quad x^2\,dx,\quad w\,dx,\quad xw\,dx,\quad u\,dx,
\]
since the coefficient pole orders are $0,4,8,6,10,9$. The equations imply
\[
dy=x^4dx,\qquad dw=(x^4+\lambda x^2)dx,
\qquad w+x\frac{dw}{dx}=(w+y)^2.
\]
Using $\Cart((A^2+xB^2)dx)=B\,dx$, we obtain
\begin{align*}
\Cart(dx)&=\Cart(x^2dx)=0,&\Cart(x\,dx)&=dx,\\
\Cart(w\,dx)&=(x^2+\rho x)dx,&\Cart(xw\,dx)&=\tau\,dx,\\
\Cart(u\,dx)&=(x^2+\sigma\tau)dx,&\Cart(\tau\,dx)&=\rho x\,dx.
\end{align*}
Consequently, the successive images are
\begin{align*}
\operatorname{im}\Cart&=\langle1,x,x^2,\tau\rangle_k\,dx,\\
\operatorname{im}\Cart^2&=\langle1,x\rangle_k\,dx,\\
\operatorname{im}\Cart^3&=\langle1\rangle_k\,dx,\qquad
\operatorname{im}\Cart^4=0.
\end{align*}
Here $\rho\ne0$ and $\tau=w+u+cxw$, so the displayed spanning sets have the asserted dimensions. Equation~\eqref{eq:EOiterates} gives
\[
\nu_6=4,\qquad\nu_4=2,\qquad\nu_2=1,\qquad\nu_1=0.
\]
The final-sequence inequalities then force $\nu_5=3$ and $\nu_3\in\{1,2\}$, so the final type is one of
\begin{equation}\label{eq:EOcandidates}
(0,1,1,2,3,4),\qquad (0,1,2,2,3,4).
\end{equation}
To distinguish these two possibilities, we show that $\nu_3=1$.

We use the de Rham description of the Dieudonn\'e module of $\operatorname{Jac}(C_\lambda)[2]$ \cite{Oda}. Let $U_0=C_\lambda\setminus\{P\}$ and $U_1=C_\lambda\setminus x^{-1}(0)$. Classes in $H^1_{\mathrm{dR}}(C_\lambda)$ are represented by triples $(\omega_0,\omega_1,f)$, where $\omega_i$ is a differential regular on $U_i$, $f$ is regular on $U_0\cap U_1$, and $\omega_0-\omega_1=df$, modulo the triples $(dg_0,dg_1,g_0-g_1)$ with $g_i$ regular on $U_i$. In these terms
\[
F(\omega_0,\omega_1,f)=(0,0,f^2),\qquad
V(\omega_0,\omega_1,f)=(\Cart\omega_0,\Cart\omega_1,0),
\]
and $H^0(C_\lambda,\Omega^1)=\ker F=\operatorname{im}V$ consists of the classes $(\omega,\omega,0)$. Choose a self-dual final filtration $(N_i)_{i=0}^{12}$ refining the canonical filtration \cite[Section~5 and Proposition~9.6]{Oort}. Thus $\dim N_i=i$, $N_{12-i}=N_i^{\perp}$, and the image under $V$ and inverse image under $F$ of each member are again members. For $0\le i\le6$, one has $\nu_i=\dim V(N_i)$. In particular, $N_6=H^0(C_\lambda,\Omega^1)$ and $N_1=\operatorname{im}\Cart^3=k\,dx$.

Put $f_0=y/x^2$, which is regular on $U_0\cap U_1$. Since $y^2=y+x^5$,
\[
f_0^2=x+\frac{y}{x^4},
\]
where $x$ is regular on $U_0$ and $y/x^4$ is regular on $U_1$, because $y$ has a pole of order $10$ at $P$. Thus the class of $f_0$ lies in the kernel of Frobenius on $H^1(C_\lambda,\mathcal O)$. Let $\xi_0\in H^1_{\mathrm{dR}}(C_\lambda)$ be any class with this image in $H^1(C_\lambda,\mathcal O)$. Adding the coboundary of $(x,y/x^4)$ to $(0,0,f_0^2)$ and using $dy=x^4dx$ gives
\[
F(\xi_0)=\bigl(dx,\ x^{-4}dy,\ 0\bigr)=(dx,dx,0),
\]
that is, $F(\xi_0)=dx$. As $\ker F=N_6$ and $F$ is $2$-linear, $F^{-1}(N_1)=N_6+k\xi_0$ has dimension $7$. Hence $N_7=N_6+k\xi_0$.

By self-duality, and because $N_6$ is maximal isotropic,
\[
N_5=N_7^{\perp}=\{\omega\in H^0(C_\lambda,\Omega^1):\langle\omega,\xi_0\rangle=0\}.
\]
For a regular differential $\omega$, the de Rham pairing $\langle\omega,\xi_0\rangle$ is the Serre-duality pairing of $\omega$ with the class of $f_0$ in $H^1(C_\lambda,\mathcal O)$. Up to a nonzero constant, it is the sum of the residues of $f_0\omega$ at the four points of $C_\lambda\setminus U_1=x^{-1}(0)$. These points are unramified over the $x$-line. Hence, for $\omega=h\,dx$ and $T=\operatorname{Tr}_{k(C_\lambda)/k(x)}$, this sum is
\[
\operatorname{Res}_{x=0}\Bigl(\frac{T(yh)}{x^2}\,dx\Bigr),
\]
the coefficient of $x$ in the polynomial $T(yh)$. Write $k(C_\lambda)=k(x,y,w)$ with $y^2+y=x^5$ and $w^2+w=y+\lambda x^3$. In each quadratic step the trace of $1$ is $0$ and the trace of the generator is $1$. Therefore
\begin{gather*}
T(y)=T(xy)=T(x^2y)=0,\qquad T(yw)=1,\qquad T(xyw)=x,\\
T(yu)=T(y^2)+cx\,T(yw)=cx,
\end{gather*}
where $u=y+cxw$ as above. The pairing with $\xi_0$ thus vanishes on $dx$, $x\,dx$, $x^2dx$ and $w\,dx$, and takes the values $1$ and $c$, up to a common nonzero constant, on $xw\,dx$ and $u\,dx$. Consequently,
\[
N_5=\langle1,x,x^2,w,u+cxw\rangle_k\,dx=\langle1,x,x^2,w,y\rangle_k\,dx.
\]

Finally, $N_3=V(N_5)=\Cart(N_5)$. Since $y=y^2+x(x^2)^2$, one has $\Cart(y\,dx)=x^2dx$. Together with the Cartier values computed above, including $\Cart(dx)=\Cart(x^2dx)=0$,
\[
N_3=\bigl\langle\Cart(x\,dx),\Cart(w\,dx),\Cart(y\,dx)\bigr\rangle_k
=\langle1,\ x^2+\rho x,\ x^2\rangle_k\,dx
=\langle1,x,x^2\rangle_k\,dx,
\]
because $\rho\ne0$. Hence $V(N_3)=\Cart(N_3)=k\,dx$ and $\nu_3=1$. This proves \eqref{eq:EOgeneric}. The final value is $3$ at $\lambda=0$ and $4$ otherwise, so the two parameter loci share no Ekedahl--Oort type.
\end{proof}

\section*{Acknowledgements}
The second author was partially supported by CNPq grant no. 302774/2025-4, FAPESP grant no. 2024/00923-6, and FAEPEX grant no. 3485/25.

\end{document}